\documentclass[10pt]{amsart}
\usepackage{amsmath, amssymb, amsthm, mathtools}
\usepackage[latin1]{inputenc} % for Macs
\usepackage{stix}
\usepackage{comment}
\usepackage[margin=1.42in]{geometry}

\newtheorem{mydef}{Definition}[section] %initiate defintion
\newtheorem{thm}{Theorem} %initiatie thm
\newtheorem{corr}{Corollary} %corollary

\newtheorem{thm2}[mydef]{Theorem} %initiatie thm

\newtheorem{prop}[mydef]{Proposition}

\newtheorem{lemma}[mydef]{Lemma}

\theoremstyle{remark}
\newtheorem{rem}[mydef]{Remark}

\newtheorem{observation}[mydef]{Observation}

\numberwithin{equation}{section}

\DeclareMathOperator{\ord}{ord}
\DeclareMathOperator{\mult}{mult}
\DeclareMathOperator{\ind}{index}
\DeclareMathOperator{\charac}{char}
\DeclareMathOperator{\wideg}{wideg}

\DeclareMathOperator{\Z}{\mathbb{Z}}
\DeclareMathOperator{\F}{\mathbb{F}}
\DeclareMathOperator{\N}{\mathbb{N}}
\DeclareMathOperator{\C}{\mathbb{C}}
\DeclareMathOperator{\Q}{\mathbb{Q}}
\DeclareMathOperator{\pind}{\iota}

\DeclareMathOperator{\K}{\mathbb{K}}
\DeclareMathOperator{\resit}{r\acute{e}sit}
\DeclareMathOperator{\MultiInd}{\alpha}

\renewcommand{\=}{\coloneqq}
\newcommand{\dd}{\hspace{1pt}\operatorname{d}\hspace{-1pt}}

\newcommand{\wD}[2]{\widehat{\Delta}_{#1}(#2)}

\begin{document}

\title{Wildly ramified power series with large multiplicity}% at the origin}
%\title{Residue fixed point index and wildly ramified power series II}
%\title{The moduli space for formal classification of parabolic maps in positive characteristic}
%\title{Towards a classification of elements of infinite order in the Nottingham group}
\author{Jonas Nordqvist}
\address{Department of Mathematics, Linnaeus University, V{\"a}xj{\"o}, Sweden}
\email{jonas.nordqvist@lnu.se}

\begin{abstract}
In this paper we consider wildly ramified power series, \emph{i.e.}, power series defined over a field of positive characteristic, fixing the origin, where it is tangent to the identity. In this setting we introduce a new invariant under change of coordinates  called the \emph{second residue fixed point index}, and provide a closed formula for it. As the name suggests this invariant is closely related to the residue fixed point index, and they coincide in the case that the power series have small multiplicity. Finally, we characterize power series with large multiplicity having the smallest possible multiplicity at the origin under iteration, in terms of this new invariant.
%In this paper we define the second residue fixed point index, which is an invariant under local change of coordinates for power series tangent to the identity defined over fields of positive characteristics. As the name suggest it is similar to the residue fixed point index. Among other results we conclude that in general there is no such thing as a normal form in the classical sense from complex dynamics over fields of positive characteristic.
\end{abstract}

\maketitle

%\tableofcontents

\section{Introduction}
%This paper is partly an extension of a previous paper by the author and Juan Rivera-Letelier \cite{NordqvistRL2019} generalizing several results from said paper, and partly novel results on wildly ramified power series.

%Iteration of functions has an essential r{\^o}le in the study of dynamics. 

Let~$\K$ be a field. We recall that a fixed point~$z_0$ of a function~$f\colon \K \to \K$ is said to be \emph{parabolic} if~$f'(z_0)$ is a root of unity. In this paper we are interested in formal power series~$f$ having a parabolic fixed point at the origin, such that~$f'(0)=1$.

% 
%Henceforth, we say that~$f$ has \emph{large} multiplicity if~$\mult(f) > \charac(\K)+1$.

We say that two elements~$f,g \in z(1+z\K[[z]])$ are \emph{formally conjugated} if there exists a formal power series~$h$ of the form~$h(0)=0, h'(0) \neq 0$ such that~$g = h^{-1}\circ f \circ h$. Clearly, conjugation is an equivalence relation on~$z(1 + z\K[[z]])$. Henceforth, we assume that~$f$ is not the identity. A property that is invariant under conjugation by~$h$ is the multiplicity of~$f$, that is the unique integer~$m$ in the expansion \[f(z)-z = a_m z^{m} +\cdots,\quad a_m \neq 0.\] The multiplicity of~$f$ (at the origin) is denoted by~$\mult(f)$, and we put~$\mult(z) \= \infty$.  Another property of~$f$ that is invariant under conjugation is the \emph{residue fixed point index} which is defined as the coefficient of~$\frac{1}{z}$ in the Laurent expansion about 0 of \[\frac{1}{z-f(z)},\] and denoted by~$\ind(f)$, see \emph{e.g.} \cite[\S 12]{Milnor2006} for a description of the residue fixed point index.\footnote{We further note that the residue fixed point index may be defined in a similar manner for any fixed point of $f$.}

%\begin{comment}% Currently not in the arXiv version
Let~$\K \=\widehat{\C}=\C\cup \{\infty\}$, then for a rational map~$\phi \in \K(z)$, not the identity, of the form~$\phi(0)=0$ and~$\phi'(0)=1$, classification under formal conjugation is described in terms of the multiplicity and the residue fixed point index (see \emph{e.g.},~\cite{Ecalle1981b, Ecalle1981,Voronin1981}). One of the main results of this paper is that this statement is false if we replace~$\K$ by a field of positive characteristic, and that the classification depends on at least one other invariant which we describe in detail in this paper. 
%\end{comment}

Let~$\K$ be a field of positive characteristic. A power series~$f$ with coefficients in~$\K$ is said to be \emph{wildly ramified} if and only if~$f(0)=0$, and~$f'(0)=1$. The group under composition formed by all wildly ramified power series is called the \emph{Nottingham group} of~$\K$, and is denoted by~$\mathcal{N}(\K)$.
Moreover, every wildly ramified power series has an associated sequence  of integers called the lower ramification numbers.  It encodes the multiplicity of the origin for the iterates of the power series. Henceforth, we say that~$f$ has \emph{large} multiplicity if~$\mult(f) > \charac(\K)+1$, and we study the lower ramification numbers of power series with large  multiplicity  at  the  origin, and give a characterization of those power series having the smallest possible lower ramification numbers. Wildly ramified power series and its lower ramification numbers have been studied in many papers, see \emph{e.g.},~\cite{Sen1969,Keating1992,LaubieSaine1998,Wintenberger2004} for classic results and background on wildly ramified power series, and~\cite{KallalKirkpatrick2019,LaubieMovahhediSalinier2002,LindahlNordqvist2018,LindahlRiveraLetelier2013,LindahlRiveraLetelier2015,Fransson2017,RiveraLetelier2003, NordqvistRL2019} for results on lower ramification numbers related to this paper.
For a background on the Nottingham group and result on the subject relating to this paper we refer to~\cite{johnson1988,Camina2000,Keating2005} and references therein. Further, as previously mentioned, one of the results of this paper concerns formal classification of parabolic maps in positive characteristic. Similar studies have been made in the case of superattracting germs in \cite{Spencer2011} and this was later further investigated in \cite{ruggiero2015}.

Endowing a field $\K$ of positive characteristic with an absolute value will necessarily make it non-archimedean. Studying the iterates of a wildly ramified power series $f$ with coefficients in $\K$ can be thought of as a discrete dynamical system, in which we are particularly interested in the fixed points of $f$ and its iterates. The topic of discrete dynamical systems over non-archimedean fields is a well-studied subject see \emph{e.g.} the monographs \cite{Silverman2007, AnashinKhrennikov2009, Benedetto2019} and references therein.

%DiscreteDynamical 

%Finally for a treatment on one dimensional dynamics over non-archimedean fields in general, see \emph{e.g.} . 

In the next section we describe our results in detail.

%For an integer~$q\geq 1$ we denote by~$\mathcal{N}_q(\K)$ the subgroup \[\mathcal{N}_q(\K) \= \left\{f \in \mathcal{N}(\K) : \mult(f) \geq q+1\right\}.\]

%Let~$f,g \in \mathbb{P}_1(\K)[[z]]$ we say that~$f$ and~$g$ are dynamically equivalent if they differ by a local change of coordinates. We say that~$f$ and~$g$ are formally (analytically)  conjugated if there is a formal (analytic) coordinate~$h$ such that \[f(z) = h^{-1} \circ g \circ h(z).\]
%Classification of formal parabolic maps under the equivalence of conjugation is easy and well-known in the case of~$\K = \widehat{\C}$. Denote by~$\sim$ the equivalence relation of formal conjugation, and let \[\mathcal{N}_q(\K) \= \left\{ f \in \K[[z]] : \mult(f) = q+1\right\}.\] Further we put \[\mathcal{M}^1_q(\K) \= \mathcal{N}_q(\K) / \sim.\] We have~$\mathcal{M}^1_q(\widehat{\C}) = \C$, since formal classification depends on the multiplicity and the residue fixed point index.

%For rational maps over the complex numbers, every parabolic map is formally classified by two invariants, namely its multiplicity and its residue fixed points index. In positive characteristic, when the multiplicity of the fixed point is greater than the ground characteristic of the field, the corresponding statement is false as manifested in Proposition \ref{conj}. In positive characteristic, the moduli space for formal classification of parabolic maps is at least three dimensional.

\section{Main results}
The main results are naturally divided into two parts. First we introduce a `new' invariant under local change of coordinates in positive characteristic. The second part concerns lower ramification numbers and in particular we provide a characterization in terms of this new invariant of wildly ramified power series with large multiplicity having the smallest possible lower ramification numbers.%%%%% Furthermore, the results entail information about formal normal forms of parabolic maps in positive characteristic.

\subsection{A positive characteristic phenomena}
In order to state the main result of the first part we define the key concept of this section, and in order to state said concept we need to introduce the following notation. Given a field~$\K$ and an integer~$q\ge 1$, denote by~$\Lambda(q,\K)$ the ordered set of integers corresponding to the set \begin{equation}\label{eq:L} \big\{\ell \in \{0, 1,\dots,q\} \, : \,  \,  \ell = q \text{ in } \K\big\}.\end{equation}
%\begin{equation}\label{eq:L}\Lambda(q,\K) \coloneqq [q]\cap [0,q].\end{equation}
The~$j$th element of~$\Lambda(q,\K)$ is denoted by~$\ell_j$, and to be precise we have the indexing~$\ell_1 \le \ell_2 \le \dots$.% | a = \ell, \ell + p, \ell + 2p,\dots, \ell + p(\lfloor q/p \rfloor-1)\right\}.\]

%\begin{equation}\label{eq:L}\Lambda(q,\K) \coloneqq [q]\cap [0,q].\end{equation}
%The~$j$th element of~$\Lambda$ is denoted by~$\ell_j$, and to be precise we have the indexing~$\ell_1 \le \ell_2 \le \dots$.% | a = \ell, \ell + p, \ell + 2p,\dots, \ell + p(\lfloor q/p \rfloor-1)\right\}.\]

\begin{rem}\label{rem1}
If~$\charac(\K) =0$ or if~$q < \charac(\K)$ then~$\Lambda(q,\K) = \{q\}.$
\end{rem}

\begin{mydef}\label{def:secres}
Let~$\K$ be a field, ~$q \geq 1$ an integer, and~$f$ a power series with coefficients in~$\K$ such that~$f(0)=0$ and~$q\=\mult(f)-1\geq1$. For each element~$\ell_j$ in~$\Lambda(q,\K)$ we denote by~$\pind_j(f)$ the coefficient of~$\frac{1}{z}$ in the Laurent expansion about 0 of%Further, denote by~$\ell$ the smallest nonnegative integer such that~$\ell = q$ in~$\K$. Furthermore, let~$f$ be a power series with coefficients in~$\K$
%Let~$\K$ be a field, and~$f$ a power series with coefficients in~$\K$, such that~$f(0)=0$, and~$q\=\mult(f)-1 \geq 1$. 
%Let~$L(q,\charac(\K))$ be as in \eqref{eq:L}, and~$\ell_i$ be its~$i$th element. We denote by~$\pind_i(f)$ the %element in~$\K$
%Let~$\ell$ be the smallest nonnegative integer such that~$\ell = q$ in~$\K$. Let~$$
%Furthermore, let~$k$ be a nonnegative integer less than or equal to~$\lceil q/p \rceil$. We define the \emph{$k$th residue fixed point index} of~$f$ as the 
%coefficient of~$\frac{1}{z}$ in the Laurent expansion about 0 of 
\[\frac{z^{q-\ell_j}}{z-f(z)}.\] In particular, in the case that~$j=1$ then we say that~$\pind_j( f)$ is the \emph{second residue fixed point index}.
\end{mydef}

\begin{rem}
As a consequence by Remark \ref{rem1}, we have that if either~$q < \charac(\K)$ or~$\charac(\K)=0$ then~~$\pind_1(f) = \ind(f)$. %In any other case we have~$\pind_{\lceil q/\charac(\K) \rceil}(f) = \ind(f)$.%This is invariant under coordinate changes by for instance \cite[Proposition 1]{NordqvistRL2019}.
\end{rem}
\begin{rem}
In the complex case for $g \in \widehat{\C}(z)$ such that $g(0)=0$, converging in a neighbourhood of the origin, the residue fixed point index is typically described by the contour integral 
\begin{equation}
  \label{eq:19}
    \ind(g) \coloneqq \frac{1}{2\pi i} \oint_C \frac{\dd z}{z - f(z)},
\end{equation}
where $C$ is a small, closed, simple curve in positive direction around the origin. The formal definition in Definition \ref{def:secres} using the theory of abstract residues (\emph{cf.} \cite[Exercise 5.10]{Silverman2007}) is equivalent in the case that $f$ is convergent in a neighbourhood of the origin.
\end{rem}
In  \S\ref{app:closed} we give a closed formula to compute $\pind_j(f)$ in terms of its first significant coefficients. For a wildly ramified power series~$f$ the residue fixed point index is invariant under local change of coordinates. This statement is also true for~$\pind_1(f)$ as manifested by the following result.
\begin{prop}
  \label{conj}
  Let~$\K$ be a field.
  Then, among power series~$f$ with coefficients in~$\K$ and satisfying~$q\=\mult(f)-1 \in [1,\infty)$, the second residue fixed point index is invariant under coordinate changes.
  That is, for every power series~$h$ with coefficients in~$\K$ such that~$h(0) = 0$ and~$h'(0) \neq 0$, the power series~$\widehat{f} \= h \circ f \circ h^{-1}$ satisfies
  \begin{displaymath}
     \pind_1(\widehat{f}) = h'(0)^{\ell-q}\pind_1(f),
  \end{displaymath}
  where~$\ell$ is the smallest nonnegative integer satisfying~$\ell = q$ in~$\K$.
\end{prop}
%\begin{rem}
%The second residue fixed point index is clearly an invariant under the action of conjugation by a wildly ramified power series.
%\end{rem}
In fact, we can say more, and the proof of Proposition \ref{conj} follows by the following generalization.
\begin{prop}\label{prop:conj2}
Let~$\K$ be a field,~$q\geq 1$ an integer, and~$f$ a power series, with coefficients in~$\K$, of multiplicity~$q+1$. If it exists, denote by~$j$ the smallest integer such that~$\pind_j(f) \neq 0$. In that case, for every power series~$h$ with coefficients in~$\K$ such that~$h(0) = 0$ and~$h'(0) \neq 0$, the power series~$\widehat{f} \= h \circ f \circ h^{-1}$ satisfies
  \begin{displaymath}
     \pind_j(\widehat{f}) = h'(0)^{\ell_j-q}\pind_j(f),
  \end{displaymath}
\end{prop}

%\begin{comment}% This is currently not in the arXiv version
\'{E}calle \cite{Ecalle1981b, Ecalle1981} and Voronin \cite{Voronin1981} independently proved that analytic classification of parabolic maps in~$\widehat{\C}(z)$ depends on an infinite number of parameters. However, as previously mentioned, the formal classification soley depends on the multiplicity and the residue fixed point index of the fixed point. Hence, Proposition \ref{conj} is an obstruction for the same situation to be true for formal equivalence in positive characteristic, and as a consequence we obtain the following  corollary.
\begin{corr}\label{c:normalforms}
%If~$q>p$, then
%\[\dim\left(\mathcal{M}^1_q\right) \geq  2.\]
In positive characteristic, the moduli space of formal classification of parabolic maps with large multiplicity is at least three dimensional.
\end{corr}
%In fact, we ask if this could be made stronger in the following question, which %would imply a sharp distinction between the characteristic zero and positive %characteristic formal classifications of parabolic maps.
%\begin{question}
%In positive characteristic, is the moduli space of formal classification of %parabolic maps infinite dimensional?
%\end{question}
%\end{comment}
%\begin{thm}

%\end{thm}

\subsection{Lower ramification numbers}
The second part of our results concerns lower ramification numbers of wildly ramified power series. Let~$n\geq0$ be an integer then the \emph{lower ramification numbers} of a wildly ramified power series~$f$ are defined as \[i_n(f) \= \mult\left(f^{p^n}\right)-1.\] 
Let~$q\= \mult(f)-1$ and denote by~$\ell$ the smallest nonnegative integer such that~$\ell = q$ in~$\K$. Then we have 
\begin{equation}\label{eq:ineqramif}i_n(f) \geq \ell(1 + p +\dots+p^{n-1}) + qp^n,\end{equation} see Proposition \ref{prop:qramifisminramif}, in \S\ref{sec:lowerramif}.
Our primary result in this section is a characterization of power series with large multiplicity, having lower ramification numbers satisfying equality in \eqref{eq:ineqramif}, \emph{i.e.} the smallest possible sequence of lower ramification numbers. We divide the results into two parts: the generic and subgeneric case.

\subsubsection{Generic case}
The following result is a characterization of power series with large multiplicity~$q+1$ having the smallest possible lower ramification numbers. In particular, among power series with multiplicity~$q+1$ these form a generic set.

\begin{thm}\label{thm:criterion}
Let~$\K$ be a field of characteristic~$p$. Further, suppose that~$q > p$ is an integer not divisible by~$p$, and denote by~$\ell$ the smallest nonnegative integer such that~$q=\ell$ in~$\K$. Then for every power series~$f \in \K[[z]]$ such that~$\mult(f) = q+1$, and every integer~$n\geq 1$, the lower ramification numbers of~$f$ are of the form \[i_n(f) = \ell(1+p+\dots+p^{n-1}) + qp^n,\] if and only if~$ \pind_1(f) \neq 0$.
\end{thm}
%\end{rem}
The proof of Theorem \ref{thm:criterion} is given in \S\ref{sec:thmA}.

In a famous paper by Sen \cite{Sen1969} we have that if~$p\mid i_0(f) = q$ then~$i_n(f) = qp^n$, which clearly satisfies equality in \eqref{eq:ineqramif}. Using this fact together with Theorem \ref{thm:criterion} and \cite[Theorem 2]{NordqvistRL2019} we obtain the following corollary, extending \cite[Corollary 1]{NordqvistRL2019} to include power series with large multiplicity.
\begin{corr}
  \label{c:genericity}
  Let~$\K$ be a field of odd characteristic.
  Then, among wildly ramified power series~$f$ with coefficients in~$\K$, those which have the smallest possible lower ramification numbers are generic.
\end{corr}
Corollary \ref{c:genericity} together with \eqref{eq:ineqramif} answers \cite[Question 1]{NordqvistRL2019}, which asks for what the lower ramification numbers for a generic power series with large multiplicity are.
\subsubsection{Subgeneric case}
The following result is a characterization of power series with large multiplicity~$q+1$ which does not necessarily attain the  smallest possible lower ramification numbers.
\begin{thm}\label{thm:criterion2}
Let~$\K$ be a field of characteristic~$p$. Further, suppose that~$q \ge p+1$ is an integer not divisible by~$p$,~$\ell$ the smallest nonnegative integer such that~$\ell = q$ in~$\K$. Finally, let~$j$ be a positive integer such that~$j \in \{1, 2,\dots, \lceil q/p \rceil-1\}$, and put~$\ell_j \= \ell + (j-1)p$. Then for every power series~$f \in \K[[z]]$ such that~$\mult(f) = q+1$, and every integer~$n\geq 1$, the lower ramification numbers of~$f$ are of the form \[i_n(f) = \ell_j(1+p+\dots+p^{n-1}) + qp^n.\] if and only if~$j$ is the smallest integer such that~$\pind_j(f) \neq 0$.
\end{thm}
We note that Theorem \ref{thm:criterion} follows from Theorem \ref{thm:criterion2}. However, as Theorem \ref{thm:criterion} is the generic case we consider it of value to be highlighted. Also the proofs of Theorem \ref{thm:criterion} and Theorem \ref{thm:criterion2} differ, since Theorem \ref{thm:criterion} follows by the invariance of $\pind_1(f)$ together with \cite[Main Lemma]{NordqvistRL2019}, and is thus self-contained. However,  Theorem \ref{thm:criterion2} does not follow as a consequence by the aforementioned lemma, and we give a direct proof depending on a known result by Laubie and Sa\"{i}ne  \cite[Corollary 1]{LaubieSaine1998}. The proof of Theorem \ref{thm:criterion2} is given in \S\ref{sec:pplus}. 

We conclude this section by defining the concept of~$q$-ramification and state our results regarding these matters.
\begin{mydef}
Let~$\K$ be a field of positive characteristic. Further, let~$f\in \K[[z]]$ be a wildly ramified series power of the form~$\mult(f) = q + 1$. If the lower ramification numbers of~$f$ satisfies \[i_n(f) = q(1 + p + \cdots + p^n),\] we say that~$f$ is~$q$-ramified.
\end{mydef}
%\begin{rem}
%The characterization of~$q$-ramified power series for~$q < p$ is given in~\cite[Theorem 2]{NordqvistRL2019}, and the case~$q > p$ answers a question posed in the same paper \cite[Question 1]{NordqvistRL2019}.
%\end{rem}
Hence, partly as a consequence of Theorem \ref{thm:criterion2} we obtain the following corollary. %***
%The following result is a characterization of~$q$-ramified power series with large multiplicity satisfying~$q \in \{p+1,\dots, 2p-1\}$. This is -- to our knowledge -- a first characterization of power series forming a subgeneric set, in terms of its lower ramification numbers.
\begin{corr}\label{qramiflarge}
Let~$\K$ be a field of odd characteristic~$p$,~$q$ an integer not divisible by~$p$, such that~$q\ge p+1$, and put~$s \= \lceil q/p \rceil$. Furthermore, let~$f \in \K[[z]]$ be a wildly ramified power series of multiplicity~$q+1$. Then~$f$ is~$q$-ramified if and only if~\[\frac{q+1}{2}- \ind(f) \neq 0,\footnote{\,The left hand side is known as the \emph{r\'{e}sidu it\'{e}ratif}, or iterative residue and is introduced in \S\ref{sec:properties}.}\] and for every positive integer~$j < s$ we have~$\pind_j(f)= 0$.
\end{corr}
%The proof of Corollary \ref{qramiflarge} is given in \S\ref{sec:pplus}.
%\begin{rem}
%As noted in a previous remark the formula is further valid for every power series such that its multiplicity is divisible by~$p$.

\subsection{Organization of the paper}
In \S\ref{sec:apps} we discuss some applications of the main results related to one-dimensional dynamical systems over non-archimedean fields. In \S\ref{sec:properties} we prove the invariance of the second residue fixed point index and its generalization $\pind_j(f)$, and we also discuss some properties of these invariants. Further, in \S\ref{sec:lowerramif} we give a proof of Theorem \ref{thm:criterion}, Theorem \ref{thm:criterion2} and Corollary \ref{qramiflarge}. In \S\ref{app:closed} we give a closed formula for $\pind_j(f)$ in terms of the first significant coefficients of the power series $f$. Finally, in Appendix \ref{app:perpoints} we give a proof of the results presented in \S\ref{sec:apps}.
%In~\S\ref{sec:properties} we discuss properties of the second residue fixed point index, and give a proof of Proposition \ref{conj}. In \S\ref{sec:lowerramif}, we discuss the results from \S\ref{sec:properties} in relation to wildly ramified power series. Finally in \S\ref{s:Main Lemma} we prove Theorem \ref{thm:criterion}.

%\section{Preliminaries}\label{sec:prel}
\section{Applications}\label{sec:apps}
As an application of our results we compute the exact minimal distance between periodic points in the open unit disk of a convergent wildly ramified power series with large multiplicity  and the origin. To state this result, we introduce some notation.
Given an non-archimedean field $(\K,|\cdot|)$, denote by $\mathcal{O}_{\K}$ its closed unit disk and by $\mathfrak{m}_{\K}$ its open unit disk respectively, \emph{i.e.}
\begin{displaymath}
  \mathcal{O}_{\K}
  \=
  \{ z \in \K : |z| \le 1 \},
  \text{ and }
  \mathfrak{m}_{\K}
  \=
  \{ z \in \K : |z| < 1 \}.
\end{displaymath}
%the ring of integers of~$\K$ and the maximal ideal of~$\mathcal{O}_{\K}$, respectively.

\begin{thm2}
  \label{thm:lower-bound}
  Let~$p$ be a prime number, let~$q\geq p+1$ be an integer not divisble by $p$, and let~$(\K, |\cdot|)$ be an non-archimedean field of characteristic~$p$.
  Furthermore, let~$f$ be a power series with coefficients in~$\mathcal{O}_{\K}$ of the form
  \[f(z)
    \equiv
    z(1 + az^q) \mod \langle z^{q+2} \rangle, \text{ with } a\neq0.\]
  Then, for every fixed point~$z_0$ of~$f$ in~$\mathcal{O}_{\K}$ that is different from~$0$ we have $|z_0|\geq |a|$, and for every periodic point~$z_0$ of~$f$ in~$\mathcal{O}_{\K}$ that is not a fixed point, we have
  \begin{equation}\label{normbound}
    |z_0|
    \ge
    |a| \cdot |\pind_1(f)|^\frac{1}{p}.
  \end{equation}
\end{thm2}

Again, we stress that we have a closed formula for $\pind_1(f)$ given in  \S\ref{app:closed}. We also note that the proof of Theorem \ref{thm:lower-bound} is almost verbatim the same proof as the proof of \cite[Theorem 3]{NordqvistRL2019} with $\resit(f)$ replaced by $\pind_1(f)$. However, even though this proof contains little to no new information, for completion, the proof is given in Appendix \ref{app:perpoints}. As a consequence of Theorem \ref{thm:lower-bound}, thus covering the case of large multiplicity, we obtain the following strengthening of \cite[Main Theorem]{LindahlRiveraLetelier2015} and \cite[Corollary 4]{NordqvistRL2019}.

\begin{corr}
  \label{c:generic isolation}
  Let~$p$ be a prime number and fix an integer~$m\geq 2$, such that $m\not\equiv 1 \pmod{p} $.
  Then, over a field of characteristic~$p$, a generic fixed point of multiplicity~$m$ is isolated as a periodic point.
\end{corr}

\section{The second residue fixed point index and its generalizations}\label{sec:properties}
This section is divided into two parts. In the first part we give a proof that the second residue fixed point index invariant under local change of coordinates, and in the case that~$\pind_1(f)=0$ the first~$\pind_j(f)$ which is nonvansihing is also invariant. In the second part we discuss some properties of the second residue fixed point index and its generalizations. Before proceeding we give some definitions that will be used throughout the paper. 

Given a ring~$R$ and elements~$a_1, \ldots, a_n$ of~$R$, denote by~$\langle a_1,\ldots,a_n \rangle$ the ideal generated by~$a_1, \ldots, a_n$.
Furthermore, denote by~$R[[z]]$ the ring of power series with coefficients in~$R$ in the variable~$z$, and denote by~$\ord_z$ the~$z$-adic valuation on~$R[[z]]$. Note that for~$f = 0$ we have~$\ord_z(0) \= +\infty$.

\subsection{Proof of Proposition \ref{conj} and \ref{prop:conj2}}
This section is devoted to prove Proposition \ref{conj} and \ref{prop:conj2}. As noted before the former follows from the latter as a special case and the proof of Proposition \ref{prop:conj2} is given after the following lemma.

\begin{lemma}\label{lemm:powersarezero}
  Let~$\K$ be a field of odd characteristic~$p$ and~$h$ a power series with coefficients in~$\K$ such that~$h(0) = 0$ and~$h'(0) \neq 0$.
  Then for every integer~$N\geq 1$ and every integer~$d \geq1$ divisible by~$p$, the coefficient~$c$ of~$\frac{1}{z}$ in the Laurent series expansion about~$0$ of
  \begin{equation}\label{zd}z^d\cdot\frac{h'(z)}{h(z)^{N+1}}\end{equation} satisfies the following properties
  \begin{enumerate}
      \item if~$N$ is \emph{not} divisible by~$p$, then~$c=0$,
      \item if~$N = d$, then~$c = h'(0)^{-d}.$
  \end{enumerate}
\end{lemma}
\begin{proof} 
To prove the first assertion, we assume that~$N$ is not divisible by~$p$ and note that \begin{equation}\label{derivative}\frac{h'(z)}{h(z)^{N+1}} = \left(-\frac{1}{N}\cdot\frac{1}{h(z)^N}\right)',\end{equation}  which is well-defined in~$\K$. The coefficient of~$\frac{1}{z}$ in the Laurent expansion about 0 of \eqref{zd} corresponds to the coefficient of~$\frac{1}{z^{d+1}}$ in the Laurent expansion about 0 of~\eqref{derivative} % \[\left(-\frac{1}{N}\cdot \frac{1}{h(z)^N}\right)',\] 
which is clearly equal to 0 since~$p\mid d$. This completes the proof of the first assertion of the lemma.

Put~\[\gamma\=h'(0), \quad h(z) \= z\left(\gamma + \sum_{j=1}^{+\infty} a_jz^j\right),\quad \text{ and } \quad b_i\= \frac{a_i}{\gamma}.\]%\quad \text{ and } \quad T_{N,d}(z) \= z^d\cdot\frac{h'(z)}{h(z)^{N+1}}.\]
For the second assertion we note that
\begin{align*}\label{expand}
z^d&\cdot \frac{h'(z)}{h(z)^{d+1}}= z^d \cdot \frac{\gamma(1 + 2b_1z + 3b_2z^2+\cdots)}{(\gamma z)^{d+1}\left(1 + b_1z + b_2z^2+\cdots\right)^{d+1}}.
\end{align*}
Denote by~$T(z)\=(1 + b_1z + b_2z^2+\cdots)^{d+1}-1$. This yields
\begin{align*}
z^d\cdot \frac{h'(z)}{h(z)^{d+1}}&= z^d \cdot \frac{\gamma(1 + 2b_1z + 3b_2z^2+\cdots)}{(\gamma z)^{d+1}\left(1  + T(z)\right)}\\
&= \frac{z^{-1}}{\gamma^d} (1 + 2b_1z  +3b_2z^2+\cdots)\left[1-T(z)+T(z)^2-\cdots\right]. \notag
\end{align*}
Hence, it follows directly that the coefficient of~$z^{-1}$ of \[z^d\cdot \frac{h'(z)}{h(z)^{d+1}},\] is~$\gamma^{-d}$. This proves the second assertion and thus the lemma.
\end{proof}
\begin{proof}[Proof of Proposition \ref{prop:conj2}]
Assume that~$j$ is the smallest integer for which~$\pind_j(f)\neq0$ and put~$\ell_j \= \ell + (j-1)p$, where~$\ell$ is the smallest integer such that~$\ell = q$ in~$\K$. In the case that~$\ell_j = q$ the proof follows from the fact that the residue fixed point index is invariant under change of coordinates. Thus, we may assume that~$\ell_j < q$. 
%If~$q =\ell$ then the proof follows by \cite[Proposition 1]{NordqvistRL2019}. Hence, we assume that~$q > \ell$ 
Further, we also put~$\gamma \= h'(0)$, and recall that~$\pind_j(f)$ is the coefficient of~$\frac{1}{z}$ in the Laurent expansion of~$\frac{z^{q-\ell_j}}{z-f(z)}$ about the origin. Put~$\Delta(z)\= f(z)-z$, and further, putting \[\left(\frac{1}{\Delta}\right)(z) \= \sum_{i=-(q+1)}^{\infty} a_i z^i,\] thus we have \[\frac{z^{q-\ell_j}}{\Delta(z)} = \sum_{i=-(q+1)}^{\infty} a_i z^{q+i-\ell_j} = \sum_{i=q-\ell_j}^q a_{-(i+1)}\frac{1}{z^{\ell_j-q+i+1}} + \sum_{i=-(q-\ell_j)}^\infty a_iz^{q-\ell_j +i}.\] Consequently,~$\pind_j(f)=a_{\ell_j-(q+1)}$, and by assumption~$a_{\ell_i +(q+1)} = 0$ for all positive integers ~$i < j$.

Put~$\widehat{\Delta}(z) \= \widehat{f}(z)-z$. Clearly~$\widehat{f}'(0) = 1$, so~$\ord_z(\widehat{\Delta}(z))\geq 2$.
Moreover, 
\begin{equation}
  \label{DeltaH}
  \begin{split}
    \Delta \circ h(z) & =
    f(h(z)) - h(z)
    \\ & =
    h(\widehat{f}(z)) - h(z) 
    \\ & =
    h(z + \widehat{\Delta}(z)) - h(z)
    \\ & \equiv
    h'(z)\widehat{\Delta}(z) \mod \langle \widehat{\Delta}(z)^2\rangle.    
  \end{split}
\end{equation}
Clearly,~$\ord_z(\Delta) = \ord_z(\widehat{\Delta}) = q + 1$ since the multiplicity is preserved under conjugation by~$h$.
%Since~$\ord_z(\Delta) = q+1$ and~$\ord_z(h') = 0$, we conclude that
%\begin{displaymath}
 % \ord_z(\Delta\circ h) = q+1
 % \text{ and }
 % {\ord_z(h'\cdot \widehat{\Delta})} = \ord_z(\widehat{\Delta}).
%\end{displaymath}
%On the other hand, by (\ref{DeltaH}) we have~$\ord_z(\Delta\circ h-h'\cdot\widehat{\Delta})\geq 2\ord_z(\widehat{\Delta})$ and therefore       
%\begin{displaymath}
%  \ord_z(\widehat{\Delta}) = \ord_z(\Delta \circ h) = q+1.
%\end{displaymath}
Thus, using (\ref{DeltaH}) we obtain
\[\Delta \circ h \equiv h'\cdot \widehat{\Delta} \mod \langle z^{2q+2} \rangle.\]
Hence, we conclude that~$\pind_j(\widehat{f})$ is equal to the coefficient of~$\frac{1}{z}$ in the Laurent series expansion about~$0$ of
\begin{equation}\label{derive}z^{q-\ell_j}\cdot\left(\frac{h'}{\Delta\circ h}\right)(z). \end{equation}

We have
\begin{displaymath}
  z^{q-\ell_j}\cdot\left(\frac{h'}{\Delta\circ h}\right)(z)
=
z^{q-\ell_j}\sum_{N = q-\ell_j}^{q} a_{- (N + 1)} \frac{h'(z)}{h(z)^{N + 1}}
+
z^{q-\ell}\sum_{i = \ell_j-q}^{+\infty} a_i h(z)^i h'(z).
\end{displaymath}
We note that the only integers in the set~$\mathcal{S} \= \{\, q-\ell_j+1,q-\ell_j+2,\dots,q\,\}$ which are divisible by~$p$ are~$q-\ell_1, q-\ell_2, \dots, q-\ell_{j-1}$, and by our assumption we have that the corresponding coefficients~$a_{\ell_i-(q+1)}=0$ for each~$i \in \{1,2,\dots, j-1\}$. Further, by the first assertion of Lemma \ref{lemm:powersarezero} we have that for each~$N$ in~$\mathcal{S}$, which is not equal to~$\ell_i$ for~$i < j$, the coefficient of~$\frac{1}{z}$ of~$z^{q-\ell_j}h'(z)h(z)^{-(N+1)}$  is zero. Consequently, by the second assertion of Lemma \ref{lemm:powersarezero} we have that the coefficient of~$\frac{1}{z}$ of \eqref{derive} is~$\gamma^{\ell_j-q}a_{\ell_j-(q+1)} = \gamma^{\ell_j-q}\pind_1(f)$. This completes the proof of the proposition.
\end{proof}

\subsection{Properties of the second residue fixed point index}
Let $\K$ be a field and $f \in \K[[z]]$ of the form $f(0)=0$ and $f'(0)=1$, then the origin is a parabolic fixed points with multiplier 1. For parabolic fixed points there is a dynamical variant of the residue fixed point index known as the \emph{r\'{e}sidu it\'{e}ratif} or iterative residue introduced by \'{E}calle in the complex setting. The iterative residue is defined by \[\resit(f) \= \frac{\mult(f)}{2} - \ind(f),\] and it behaves nicely under iteration as for a positive integer $n\not\equiv 0$ in $\K$ we have \[\resit(f^n) = \frac{1}{n}\resit(f),\] see for instance \cite[\S 12]{Milnor2006} for the characteristic 0 case, and \cite[Appendix A]{NordqvistRL2019} for the positive characteristic case. In positive characteristic the second residue fixed point index shares this iterative property with the iterative residue as manifested in the following proposition.
\begin{prop}\label{prop:iter}
 Let~$p$ be a prime and~$\K$ a field of characteristic~$p$. Let~$f \in \K[[z]]$ be a wildly ramified power series of finite multiplicity such that~$\mult(f)  \geq p+1$. Then for any positive integer~$n$ not divisible by~$p$, we have \[\pind_1(f^n) = \frac{1}{n}\pind_1(f).\]
\end{prop}

The proof of Proposition \ref{prop:iter} is given at the end of this section. The following lemma will be useful.

%\begin{prop}\label{thm:explicit}
%Let~$\K$ be a field of positive characteristic~$p$. Furthermore, let~$q\geq1$ be an integer not divisible by~$p$ and suppose that~$f$ is a power series with coefficients in~$\K$, such that~$\mult(f) = q+1$. Then the second residue fixed point index of~$f$ is given by the following formula
%\begin{equation}
%    \label{eq:closed-formula}\pind_1(f) = - \frac{1}{a_q^{\ell(q+1)/q}}\sum_{\substack{\MultiInd \in \N^{\ell+1} \\ |\MultiInd| = \ell, \| \MultiInd \| = \ell }} (-1)^{\ell-\MultiInd_0}\binom{\ell-\MultiInd_0}{\MultiInd_{1},\ldots,\MultiInd_{\ell}}\prod_{j = 0}^\ell a_{q + j}^{\iota_j}.\end{equation}

%\end{prop}

%The proof of Proposition \ref{thm:explicit} is given in Appendix \ref{app:proof}.

\begin{lemma}[Lemma 1, \cite{NordqvistRL2019}]\label{removeterms}
  Let~$\K$ be a field,~$q \ge 1$ an integer, and~$f$ a power series with coefficients in~$\K$ of the form
  \begin{displaymath}
    f(z)
    =
    z\left(1 + \sum_{j=q}^{+\infty} a_jz^j\right), \text{ with } a_q \neq 0.
  \end{displaymath}
  Then, for every integer~$k \ge 1$ such that~$a_{q + k} \neq 0$ and~$k\neq q$ in~$\K$, there is a polynomial~$h$ such that~$\mult(h) = k + 1$ and
  \[ h \circ f \circ h^{-1}(z) \equiv f(z) - a_{q+k}z^{q+k+1} \mod \langle z^{q+k+2}\rangle.\]
\end{lemma}

\begin{lemma}\label{prop:alphaa}
Let~$\K$ be a field of positive characteristic~$p$. Furthermore, let~$q\geq1$ be an integer not divisible by~$p$, and denote by~$j$ the smallest integer such that~$\pind_j(f)\neq 0$. Furthermore, suppose that~$f$ is a power series with coefficients in~$\K$, of the form \[f(z) \equiv z\left(1 + \alpha z^q + \beta z^{q+\ell_j}\right)\mod \langle z^{q+\ell_j+2} \rangle, \quad \alpha\neq 0.\] Then~$\pind_j(f) = \frac{\beta}{\alpha^2}$.
\end{lemma}

\begin{proof}
Recall that~$\pind_j(f)$ is the coefficient of~$\frac{1}{z}$ in the Laurent expansion of \begin{equation}\label{zql}\frac{z^{q-\ell_j}}{z-f(z)}\end{equation} about the origin. Expanding \eqref{zql} about 0 we obtain \[\frac{z^{q-\ell_j}}{z-f(z)} = -\frac{z^{q-\ell_j}}{\alpha z^{q+1}}\frac{1}{1+\frac{\beta}{\alpha} z^{\ell_j} +\cdots } = -\frac{1}{\alpha}z^{-\ell_j-1}\left(1 -\frac{\beta}{\alpha} z^{\ell_j}-\cdots\right).\] Clearly, the coefficient of~$\frac{1}{z}$ in \eqref{zql} is~$\frac{\beta}{\alpha^2}$, which proves the proposition.
\end{proof}

\begin{lemma}\label{prop:ind2conj}
Let~$p$ be a prime number and~$\K$ a field of characteristic~$p$. Moreover, let~$f$ be a power series with coefficients in~$\K$ such that \[f(z) = z(1 + \alpha z^q + \cdots), \quad \alpha\neq0,\] and let~$j$ be the smallest integer such that~$\pind_j(f)\neq0$. Then,~$f$ is conjugated to a power series with coefficients
in~$\K$, of the form
\[z\mapsto z\left(1 + \alpha z^q + \beta z^{q+\ell_j}\right) \mod \langle z^{q+\ell_{j}+p+1} \rangle,\] such that~$\pind_j(f) = \beta/\alpha^2$.
\end{lemma}

\begin{proof}
%  Note that the power series~$\widehat{f}(z) \= \gamma^{-1} f(\gamma z)$ satisfies~$\mult(\widehat{f}) = q + 1$ and that the coefficient of~$z^{q + 1}$ in~$\widehat{f}$ is equal to~$1$.
%Denote by~$\alpha$ the coefficient of~$z^{q+1}$ in~$f$ is~$\alpha$. 
%First, we note that if~$q > p$ then~$2q>q+2\ell$ since~$q-\ell\geq p$ and~$\ell<p$, so that~$2q+1$ is not in the set~$\{q+2,\cdots, q+2\ell\}$. Consequently, 
We utilize Lemma \ref{removeterms} repeatedly  to eliminate each term of  degrees in the set \[\mathcal{S} = \{q+2, q+3,\dots, q+\ell_{j+1}\}\setminus \{q + \ell_1+1, q + \ell_2+1,\dots, q+\ell_j+1\}.\] Then there are elements~$\beta_1,\dots, \beta_{j}$ in~$\K$ such that~$f$ is conjugated to a power series~$g$ of the form \[g(z) \equiv  z\left(1 + \alpha z^q + \beta_1z^{q+\ell_1} + \beta_2z^{q+\ell_2}+\dots + \beta_{j}z^{q+\ell_{j}}\right) \mod \langle z^{q + \ell_{j}+p+1} \rangle.\] By Lemma \ref{prop:ind2conj}~$\pind_1(g) = \beta_1/\alpha^2$, which is zero by assumption implying~$\beta_1=0$, and by inductively applying Lemma \ref{prop:ind2conj} we obtain~$\beta_1 = \dots = \beta_{j-1} = 0$.
%\[\mathcal{S}\= \left\{q+2,\dots, q+\ell,q+\ell+2,\dots, q+\ell_{j+1}\right\},\] of~$f$, such that the degree is not congruent to~$q$ modulo~$p$. 
Further, denote by~$h_k$ the coordinate which eliminate the coefficient of degree~$k$ of~$f$. Then, by Lemma \ref{removeterms}, for all~$k\in \mathcal{S}$ we have~$h_k'(0)=1$. Put \[h \= h_{q+\ell_{j+1}} \circ\dots\circ h_{q+2},\] and consequently~$h'(0)=1$. Thus, there exists an element~$\beta \in \K$ such that  \[h \circ f \circ h^{-1}(z) \equiv z\left(1 + \alpha z^q + \beta z^{q+\ell_j}\right) \mod \langle z^{q+\ell_j+p+1} \rangle.\] The final claim that~$\pind_j(f) = \beta/\alpha^2$ follows by Lemma \ref{prop:alphaa}. This completes the proof of the lemma.
\end{proof}

%Lemma \ref{removeterms} allows for us to kill all the intermediate terms in between~$z^{q+1}$ and~$z^{q+\ell+1}$ in \[f(z) \equiv z\left(1 + \sum_{i=q}a_i z^i\right),\] by conjugation. Furthermore, seeing that each~$h$ have multiplier 1, we have that~$f$ is conjugated to \[z\mapsto z\left(1 + \alpha z^q + \beta z^{q+\ell}\right) \mod \langle z^{q+\ell+2}\rangle.\] Hence, expanding~$\frac{z^{q-\ell}}{z-f(z)}$ about the origin we obtain \begin{equation}\label{zql2}\frac{z^{q-\ell}}{z-f(z)} = -\frac{z^{q-\ell}}{\alpha z^{q+1}}\frac{1}{1+\frac{\beta}{\alpha} z^\ell +\cdots } = -\frac{1}{\alpha}z^{-\ell-1}\left(1 -\frac{\beta}{\alpha} z^\ell-\cdots\right).\end{equation} Clearly, the coefficient of~$\frac{1}{z}$ in \eqref{zql2} is~$\frac{\beta}{\alpha^2}$, and thus~$\pind_1(f) = \frac{\beta}{\alpha^2}$.

Before we give the proof of Proposition \ref{prop:iter} we make the following elementary observation.
\begin{observation}\label{obs:1}
For any power series~$\psi$ having the origin as a fixed point of multiplicity~$q+1$ we have for every integer~$n\geq0$,
\begin{align*}
\psi^n(z)-z \equiv n (\psi(z) -z)  \mod \langle z^{2q+1} \rangle.
\end{align*} 
\end{observation}
\begin{proof}[Proof of Proposition \ref{prop:iter}]

Let~$\alpha \in \K^\times$,~$\beta \in \K$, and put~$\beta/\alpha^2\=\pind_1(f)$,~$q\=\mult(f)-1$, and denote by~$\ell$ the smallest nonnegative integer such that~$\ell = q$ in~$\K$. We may assume by Lemma \ref{prop:ind2conj} that~$f$ is of the form \[f(z) \equiv z(1 + \alpha z^q + \beta z^{q+\ell}) \mod \langle z^{q+\ell+2} \rangle.\]

Hence, we have \begin{equation}\label{fn}f^n(z) \equiv z(1+ n\alpha z^{q} + n\beta z^{q+\ell}) \mod \langle z^{q+\ell+2}\rangle,\end{equation} Thus, by Lemma \ref{prop:alphaa} and \eqref{fn} we have \[\pind_1(f^n) = \frac{n\beta}{(n\alpha)^2} = \frac{1}{n}\pind_1(f).\] This completes the proof of the proposition.
\end{proof}

%\begin{proof}

%\end{proof}
\section{Lower ramification numbers for power series with large multiplicity}\label{sec:lowerramif}
In this section we first prove the lower bound for the sequence of lower ramification numbers for power series in which the origin as a fixed point have multiplicity which is larger than the ground characteristic of the field. In the second part, we give a proof of Theorem \ref{thm:criterion} and in the last part we prove Theorem \ref{thm:criterion2} and Corollary \ref{qramiflarge}. %discuss a particular example which was raised as a question in \cite[Question 2]{NordqvistRL2019}.

\subsection{Lower bound of lower ramification numbers}\label{sec:deltaqraminf}
This section is devoted to prove the following proposition. For completion, we acknowledge that it is possible to deduce the proposition using \cite[Theorem 6]{Camina2000}. %Nonetheless, before proceeding we make the following definition, used in~\cite[\emph{Exemple}~3.19]{RiveraLetelier2003} and~\cite{LindahlRiveraLetelier2013}, that will be used throughout this section and the next. Given a wildly ramified power series~$f$~put~$\Delta_0 \= \text{id}$, and for any integer~$m\ge1$ we define recursively \[\Delta_m(z) = \Delta_{m-1}(f(z)) - \Delta_{m-1}(z).\] Given a prime~$p$ we will make use several times of the fact that \[\Delta_p(z) = f(z)-z \mod \langle p \rangle.\]
\begin{prop}\label{prop:qramifisminramif}
  Let~$\K$ be a field of characteristic~$p$. Further, suppose that~$q\geq1$ is an integer and denote by~$\ell$ the smallest nonnegative integer such that~$q=\ell$ in~$\K$. Then for every power series~$f$ in~$\K[[ z]]$ satisfying~$\mult(f) = q+1$, we have for every integer~$n \ge 1$
  \begin{equation}
    \label{ineq}
    i_n(f) \geq \ell(1 + p + \cdots + p^{n-1}) + qp^n.
  \end{equation}
\end{prop}
In the proof of Proposition \ref{prop:qramifisminramif} we make several times use of \cite[Theorem 1]{Sen1969} referred to as Sen's theorem, which states that if~$f$ is a wildly ramified power series, and~$i_m(f)$ is finite for every integer~$m\geq0$, then for every integer ~$n\geq1$ we have \[i_n(f) \equiv i_{n-1}(f) \pmod{p^n}.\] 
Further in this and proceeding sections we make several times use of the `$\Delta$'-operators for a wildly ramified power series~$f$, introduced in \cite[\S3.2]{RiveraLetelier2003}, and \cite{LindahlRiveraLetelier2013}, of defining recursively for every integer~$m\geq 0$ the power series~$\Delta_m$, by~$\Delta_0( z)\=z$ and for~$m \ge 1$, by
\begin{equation}\label{eq:Delta}\Delta_m( z) \= \Delta_{m-1}(\widehat{f}( z)) - \Delta_{m-1}( z).\end{equation}  In particular we make several times use of the fact that for a prime~$p$ we have \begin{equation}\label{eq:7}\Delta_p(z) = f^p(z)-z \mod \langle p \rangle.\end{equation}
The proof of Proposition \ref{prop:qramifisminramif} follows after the following lemma.
\begin{lemma}\label{lemma:delta}
Let~$\K$ be a field of positive characteristic~$p$, and~$\psi(z) \in z\left(1+z\K[[z]]\right)$, and define by~$\Delta_0(z) \= z$, and for integers~$m\geq1$ put \[\Delta_m(z) = \Delta_{m-1}(\psi(z)) - \Delta_{m-1}(z).\] Then \[\ord_z(\Delta_m)-\ord_z(\Delta_{m-1}) \geq \ord_z(\Delta_1)-1.\]
\end{lemma}

\begin{proof}
Put~$q\=\ord_z(\Delta_1)-1$, and let~$\psi(z) = z\left(1 + \sum_{i\geq q}b_iz^i\right)$. Further assume that~$\Delta_m(z)=\sum_{i\geq r}a_iz^i$. Then \[\Delta_{m+1}(z) = \sum_{i\geq r} a_iz^i\left[(1+b_qz^q+\cdots)^i - 1\right].\] Clearly, if~$p\nmid r$ we have~$\ord_z(\Delta_{m+1}) = r+q$, and if~$p\mid r$ we obtain~$\ord_z(\Delta_{m+1}) > r+q$ hence \begin{equation*}\label{eq:deltis}\ord_z(\Delta_{m+1})-\ord_z(\Delta_m) \geq q,\end{equation*} which proves the assertion.
\end{proof}

\begin{proof}[Proof of Proposition \ref{prop:qramifisminramif}]
%We use the strategy of~$\Delta$-operators introduced in~\cite[\S3.2]{RiveraLetelier2003} and~\cite{LindahlRiveraLetelier2013}, of defining recursively for every integer~$m\geq 0$ the power series~$\Delta_m$, by~$\Delta_0( z)\=z$ and for~$m \ge 1$, by
%\[\Delta_m( z) \= \Delta_{m-1}(\widehat{f}( z)) - \Delta_{m-1}( z).\] In the proof we make several times use of the fact~$\Delta_p(z) = f^p(z)-z$, 
We utilize the~$\Delta$-operator, and note that~$\Delta_p(z) = f^p(z)-z$ implies that~$i_1(f) = \ord_z(\Delta_p)-1$. Using Lemma \ref{lemma:delta} together with~$\ord_z(\Delta_1) = q+1$, we obtain~$\ord_z(\Delta_p) \geq qp+1$. However, by Sen's theorem we have~$\ord_z(\Delta_p)\equiv \ord_z(\Delta_1) \equiv \ell+1 \pmod{p}$. Hence, \begin{equation}\label{firstinduc}i_1(f) \geq qp + \ell.\end{equation}

We proceed by induction, and assume that \eqref{ineq} holds for some integer~$n\geq1$, and put~$g(z) \= f^{p^n}(z)$. Define by~$\wD{1}{z} \=g(z)-z$, and for integers~$m\geq 2$ we put
\[\wD{m}{z}
  =
  \wD{m-1}{g(z)}-\wD{m-1}{z}.\]
Again, we note that~$\wD{p}{z} = g^p(z)-z = f^{p^{n+1}}(z)-z$, and hence~$i_{n+1}(f) = \ord_z(\wD{p}{z})-1$.
By the induction assumption and Lemma \ref{lemma:delta} we have
\[\ord_z(\widehat{\Delta}_m)-\ord_z(\widehat{\Delta}_{m-1})
  \geq
  \ell (1 + p + \cdots + p^{n-1}) + qp^n. \]
Thus, by an induction argument we obtain~$\ord_z(\widehat{\Delta}_p) \geq p\left[\ell  (1 + p + \cdots + p^{n-1})+qp^{n}\right]+1$. Again, by Sen's theorem we must have~$\ord_z(\widehat{\Delta}_p) \equiv \ord_z(\widehat{\Delta}_1) \equiv \ell(p + p^2 + \cdots + p^{n})+1 \pmod{p^{n+1}}$. Hence, we obtain
\begin{align*}i_{n+1}(f) + 1
  =
  \ord_z(\widehat{\Delta}_p)
  &\geq
  \ell p (1 + p + \cdots + p^{n-1}) + qp^{n+1} +\ell + 1\\
  &=
  \ell(1 + p + \cdots + p^{n}) + qp^{n+1} + 1,\end{align*}
which proves the induction step and together with \eqref{firstinduc}, the proof of the proposition.
\end{proof}

\subsection{Proof of Theorem \ref{thm:criterion}}\label{sec:thmA}
This section is devoted to the proof of Theorem \ref{thm:criterion}, which essentially depends on two  results, one which is a special case of \cite[Main Lemma]{NordqvistRL2019}, stated below as Lemma \ref{lemma:main}. The other is the computation of the first significant terms of~$p$-power iterates of wildly ramified power series with large multiplicity, stated as Proposition \ref{prop:pn}. %The proof of the latter result is given at the end of this section.
\begin{lemma}[Main Lemma, \cite{NordqvistLindahl2017}]\label{lemma:main}
  Let~$p$ be an odd prime number, and consider the rings \begin{displaymath}
  \Z_{(p)} \= \left\{\frac{m}{n} \in \Q : m, n \in \Z, p\nmid n \right\},
  \end{displaymath}
  \begin{displaymath}
      F_1 \= \Z_{(p)}[x_0, x_1],
  \text { and }
  F_\infty \= \Z_{(p)}[x_0, x_1,x_2,\ldots].
\end{displaymath}
   Moreover, let~$q \ge p+1$ be an integer that is not divisible by~$p$, and~$\ell \ge 1$ an integer satisfying
  \begin{displaymath}
    \ell \equiv q \pmod{p},
    \text{ and }
    \ell \le p - 1
    \text{ or }
    2\ell + 1 \le q.
  \end{displaymath}
  Then the power series~$\widehat{f}$ in~$F_\infty[[z]]$ defined by
\[\widehat{f}(z) \= z\left(1 + x_0z^{q} + x_1z^{q+\ell} + z^{q+2\ell}\sum_{i=1}^\infty x_{i+1}z^{i}\right), \]
satisfies the following property: There are~$\beta$ and~$\gamma$ in~$F_1$ such that
\begin{align}
  \label{e:main beta}
  \beta
  & \equiv
    - x_0^{p - 1} x_1 \mod p F_1,\\
      \label{e:main gamma}
      \gamma
  & \equiv
      - x_0^{p - 2} x_1^2 \mod p F_1,
      \intertext{ and }
      \widehat{f}^p(z)
  & \equiv
    z \left(1 + \beta z^{qp + \ell} + \gamma z^{qp + 2\ell} \right) \mod \langle p, z^{qp + 2\ell + 2} \rangle.
\end{align}
\end{lemma}

\begin{prop}\label{prop:pn}
Let~$p$ be a prime and~$\K$ a field of characteristic~$p$. Furthermore, let~$q\ge p+1$ be an integer not divisible by~$p$, and denote by~$\ell$ the smallest nonnegative integer such that~$q=\ell$ in~$\K$, and put \[\omega(n)\coloneqq  \ell(1 + p + \dots + p^{n-1}) + qp^n.\] Let~$g$ be a power series with coefficients in~$\K$ of the form \begin{equation}\label{eq:g}g(z)\equiv z(1 + \alpha z^q + \beta z^{q+\ell}) \mod \langle z^{q+2\ell +1} \rangle.\end{equation} Then for each integer~$n\geq 0$ we have 
\begin{align}\label{eq:indprop}g^{p^n}(z) - z &\equiv (-1)^n\alpha^{\frac{p^{n+1}-1}{p-1}}\pind_1(f)^{\frac{p^n-1}{p-1}}z^{\omega(n)+1}\notag\\ &\quad+ (-1)^n\alpha^{\frac{p^{n+1}-1}{p-1}+1}\pind_1(f)^{\frac{p^n-1}{p-1}+1}z^{\omega(n)+\ell + 1} \mod \langle z^{\omega(n)+\ell + 2}\rangle.\end{align}
In particular,~$i_n(g) = \omega(n)$ if and only if~$\pind_1(g)\neq 0$.
\end{prop}
%We continue to prove Theorem \ref{thm:criterion} assuming Lemma \ref{lemma:main} and Proposition \ref{prop:pn}.

\begin{proof}%[Proof of Proposition \ref{prop:pn}]
Put \[\widehat{g}(z) \= z\left(1 + x_0 z^{q} + x_1 z^{q+\ell} +  z^{q+2\ell}\sum_{i=1}^\infty x_{i+1} z^{i}\right),\] and let~$\Z_{(p)}$,~$F_1$ and~$F_\infty$ be defined as in Lemma \ref{lemma:main}. Moreover, let~$h \colon F_\infty \to \K$ be the unique ring homomorphism extending the reduction map~$\Z_{(p)} \to \F_p$, such that~$h(x_0) = \alpha$,~$h(x_1) = \beta$ and such that for every~$i \ge 2$ the element~$h(x_i)$ of~$\K$ is the coefficient of~$z ^{q+2\ell + i-1}$ in~$g$.
Then~$h$ extends to a ring homomorphism~$F_{\infty}[[ z  ]] \to \K [[ z  ]]$ that maps~$\widehat{g}$ to~$g$. Hence, by Lemma \ref{lemma:main} we obtain \begin{align}\label{eq:n1}g^p(z) - z &\equiv -\alpha^{p-1}\beta z^{qp+\ell + 1} - \alpha^{p-2}\beta^2z^{qp+2\ell +1} \mod \langle  z^{2p+2\ell+2} \rangle \notag\\
&\equiv -\alpha^{p+1}\pind_1(f) z^{qp+\ell + 1} - \alpha^{p+2} \pind_1(f)^2z^{qp+2\ell +1} \mod \langle z^{2p+2\ell+2} \rangle.
\end{align}
Thus,~$i_1(g) = i_1(f) = qp +\ell$ if and only if~$\pind_1(f)\neq 0$. We proceed by induction in~$n$, assume that for some integer~$n\geq1$ we have \begin{align}\label{eq:ind}g^{p^n}(z) - z &\equiv (-1)^n\alpha^{\frac{p^{n+1}-1}{p-1}}\pind_1(f)^{\frac{p^n-1}{p-1}}z^{\omega+1}\notag\\ &\quad+ (-1)^n\alpha^{\frac{p^{n+1}-1}{p-1}+1}\pind_1(f)^{\frac{p^n-1}{p-1}+1}z^{\omega+\ell + 1} \mod \langle z^{\omega+\ell + 2}\rangle.\end{align} This is clearly true in the case that~$n=1$ by \eqref{eq:n1}. Then by Lemma \ref{prop:ind2conj}, there are~$A, B \in \K$ such that~$g^{p^n}$ is conjugated to a power series~$G$ of the form \[G(z) \equiv z(1+ Az^\omega + Bz^{\omega+\ell}) \mod \langle z^{\omega + 2\ell +1}\rangle,\] such that 
\[A = (-1)^n\alpha^{\frac{p^{n+1}-1}{p-1}}\pind_1(f)^{\frac{p^n-1}{p-1}}, \quad \text{ and } \quad B = (-1)^n \alpha^{\frac{p^{n+1}-1}{p-1}+1}\pind_1(f)^{\frac{p^n-1}{p-1}+1}.\] 
Hence, we have~$B = \alpha\pind_1(f)A$. Note further that the conjugation mapping~$g^{p^n}$ to~$G$ is a composition of coordinates~$h_k$ eliminating terms of degree~$k\geq \omega + \ell +2$. Thus, let~$h$ be the composition of all~$h_k$, \emph{i.e.}, the coordinate taking~$g^{p^n}$ to~$G$ then~$\mult(h) \geq \ell + 2$. Further, put \[\widehat{G}(z) \= z\left(1 + y_0 z^{\omega} + y_1 z^{\omega+\ell} +  z^{\omega+2\ell}\sum_{i=1}^\infty x_{i+1} z^{i}\right).\] Repeating the same argument as for the case~$n=1$ and specializing~$\widehat{G}$ to~$G$, then by Lemma \ref{lemma:main} using~$\widehat{G} = \widehat{f}$  we obtain 
\begin{align}\label{eqGp}
G^p(z)-z&\equiv -A^{p-1}B z^{\omega p +\ell+1} - A^{p-2}B^2z^{\omega p + 2\ell + 1}\mod \langle  z^{\omega p + 2\ell +2} \rangle\notag\\
&\equiv -\alpha \pind_1(f)A^{p} z^{\omega p +\ell+1} - \alpha^2\pind_1(f)^2A^{p}z^{\omega p + 2\ell + 1}\mod \langle  z^{\omega p + 2\ell +2} \rangle\notag\\
%&\equiv -\alpha \pind_1(f)A^{p} z^{\omega p +\ell+1} - \alpha^2\pind_1(f)^2A^{p}z^{\omega p + 2\ell + 1}\mod \langle  z^{\omega p + 2\ell +2} \rangle\notag\\
&\equiv (-1)^{n+1}\alpha^{\frac{p^{n+2}-1}{p-1}}\pind_1(f)^{\frac{p^{n+1}-1}{p-1}}z^{\omega p+\ell + 1} \\&\quad+ (-1)^{n+1}\alpha^{\frac{p^{n+2}-1}{p-1}+1}\pind_1(f)^{\frac{p^{n+1}-1}{p-1}+1}z^{\omega p+2\ell + 1} \mod \langle z^{\omega p+2\ell + 2}\rangle.\notag
\end{align}
Finally, we recall that the coordinate~$h$ that maps~$g^{p^n}$ to~$G$ satisfies~$\mult(h) \geq \ell+2$, and thus we have \[h^{-1}\circ G^p \circ h(z) \equiv g^{p^{n+1}}(z) \mod \langle z^{\omega p + \ell + \mult(h)} \rangle.\] Hence, \eqref{eqGp} proves the desired induction assumption \eqref{eq:ind}, which proves the proposition.
\end{proof}

\begin{proof}[Proof of Theorem \ref{thm:criterion}]
 By our hypothesis that~$q\ge p+1$ the proof follows immediately by  Lemma \ref{prop:ind2conj} and Proposition \ref{prop:pn} since the former implies that~$f$ is conjugated to a power series~$g$ of the form \eqref{eq:g}.
\end{proof}

\subsection{Proof of Theorem \ref{thm:criterion2} and Corollary \ref{qramiflarge}}\label{sec:pplus}
In this section we give a proof of Theorem \ref{thm:criterion2}, and its corollary. We make use of the following result.%The proof of Theorem \ref{thm:criterion2} differs from that of Theorem \ref{thm:criterion} since in the latter case we make a self-contained 

\begin{prop}[\cite{LaubieSaine1998}, Corollary~1]
  \label{laubiecorr}
Let~$p$ be a prime number,~$\K$ a field of characteristic~$p$, and~$f$ in~$\K[[z ]]$ such that~$f(0)=0$ and~$f'(0)=1$.
If
\begin{displaymath}
  p\nmid i_0(f)
  \text{ and }
  i_1(f) < (p^2-p+1)i_0(f), 
\end{displaymath}
then for every integer~$n\geq 1$ we have
\[ i_n(f) = i_0(f) + (1 + p + \cdots + p^n)(i_1(f)-i_0(f)). \]
\end{prop}
%Further, in this section we restrict ourselves to work over finite fields, although we note that the results could extend to any field of positive characteristic.
%In \cite{NordqvistRL2019} the authors raise the question to when a wildly ramified series is~$(p+1)$-ramified. Corollary \ref{c:genericity} states that a generic power series satisfying the necessary condition~$\mult(f) = p + 2$ is \emph{not}~$(p+1)$-ramified.

The proof of Theorem \ref{thm:criterion2} is given after the following proposition in which we compute the first significant terms of the $p$th iteration of a power series with large multiplicity at the origin.

% \begin{lemma}[Lemma 5, \cite{NordqvistRL2019}]\label{wilson}
%   Let~$p$ be an odd prime number,~$a$ and~$b$ in~$\F_p$ such that~$a\neq 0$, and let~$w \colon \F_p \to \F_p$ be defined by~$w(n) \= an +b$.
%   Denoting~$s' \= -a^{-1}b$, we have
%   \begin{displaymath}
%     \prod_{s \in \F_p \setminus \{s'\}} w(s) = -1
%   \text{ and }
%     \sum_{s \in \F_p \setminus\{s'\}} \frac{1}{w(s)} = 0.
%   \end{displaymath}
% \end{lemma}

\begin{prop}\label{prop:deltaiteratesshort}
Let~$p$ be a prime number and consider the rings~$\Z_{(p)}, F_1$ and~$F_\infty$ defined in Lemma \ref{lemma:main}.
%\begin{displaymath}
%   \Z_{(p)} \= \left\{\frac{m}{n} \in \Q : m, n \in \Z, p\nmid n \right\},
%   \end{displaymath}
%   \begin{displaymath}
%       F_1 \= \Z_{(p)}[x_0, x_1],
%   \text { and }
%   F_\infty \= \Z_{(p)}[x_0, x_1,x_2,\ldots].
% \end{displaymath}
%   Moreover, let~$q \ge p+1$ be an integer that is not divisible by~$p$, and~$\ell \ge 1$ an integer satisfying \[\ell \equiv q \pmod{p}, \text{ and } \ell \le q.\]
%   Then the power series~$\widehat{f}$ in~$F_\infty[[z]]$ defined by
% \[ \widehat{f}(z)
%   \=
%   z\left(1 + x_0 z^q + x_1z^{q + \ell} + z^{q + \ell}\sum_{i=1}^{+ \infty} x_{i+1}z^i\right), \]
% satisfies
% \begin{displaymath}
%   \widehat{f}^p(z)
%   \equiv
%   z \left(1 + \beta z^{q(p + 1)}\right) \mod \langle p, z^{q(p + 1) + 2} \rangle,
% \end{displaymath}
% where 
% \[
%     \beta \equiv 
%         \begin{cases}
%             x_0^{p - 1} \left( x_0^2 \frac{q + 1}{2} - x_1 \right) \mod pF_1 \quad \text{ if~$\ell = q$} \\
%             -x_0^{p-1}x_1 \mod pF_1 \quad \text{ if~$\ell < q$.}
%         \end{cases}.
% \]
% \end{prop}
Then for each integer~$q \ge p+1$ not divisible by~$p$, and each integer~$\ell_j$ in~$\Lambda(q,\F_p)$, the power series~$\widehat{f}$ in~$F_\infty[[z]]$ defined by
\[ \widehat{f}(z)
  \= 
  z\left(1 + x_0 z^q + x_1z^{q+\ell_j} + z^{q+\ell_{j}+p}\sum_{i=1}^{+ \infty} x_{i+1}z^i\right), 
  \]
satisfies
\begin{displaymath}
  \widehat{f}^p(z)-z
  \equiv
  \beta z^{qp + \ell_j + 1} \mod \langle p, z^{qp + \ell_j + 2} \rangle
\end{displaymath}
where~$\beta \in F_1$ satisfies \[\beta \equiv \begin{cases}\begin{array}{lll} - x_0^{p - 1}x_1 & \mod  pF_1  &  \text{ if~$\ell_j < q$}\\
x_0^{p - 1} \left( x_0^2 \frac{q + 1}{2} - x_1 \right) & \mod pF_1 & \text{ if~$\ell_j=q$}.\end{array}\end{cases}\] 
\end{prop}

% \begin{lemma}[Lemma 4, \cite{NordqvistRL2019}]\label{wilson}
%   Let~$p$ be an odd prime number,~$a$ and~$b$ in~$\F_p$ such that~$a\neq 0$, and let~$w \colon \F_p \to \F_p$ be defined by~$w(n) \= an +b$.
%   Denoting~$s' \= -a^{-1}b$, we have
%   \begin{displaymath}
%     \prod_{s \in \F_p \setminus \{s'\}} w(s) = -1
%   \text{ and }
%     \sum_{s \in \F_p \setminus\{s'\}} \frac{1}{w(s)} = 0.
%   \end{displaymath}
% \end{lemma}

% \begin{prop}\label{prop:deltaiteratesshort}
% Let~$p$ be an odd prime number and consider the rings
% \begin{displaymath}
%   \Z_{(p)} \= \left\{\frac{m}{n} \in \Q : m, n \in \Z, p\nmid n \right\},
%   \end{displaymath}
%   \begin{displaymath}
%       F_1 \= \Z_{(p)}[x_0, x_1],
%   \text { and }
%   F_\infty \= \Z_{(p)}[x_0, x_1,x_2,\ldots].
% \end{displaymath}

% The proof of Theorem~\ref{thm:criterion2} is given after the proof of this proposition.
The proof of Proposition \ref{prop:deltaiteratesshort} is given at the end of this section.

\begin{proof}[Proof of Proposition~\ref{prop:deltaiteratesshort}]
In the case that~$\ell_j = q$ then the proof follows from \cite[Proposition 5]{NordqvistRL2019}. Thus we may assume that~$\ell_j < q$. We will utilize the~$\Delta$-operators, as defined in \S\ref{sec:deltaqraminf}, \emph{i.e.} define recursively for every integer~$m\geq 1$ the power series~$\Delta_m$, by~$\Delta_1(z)\=\widehat{f}(z)-z$ and for~$m \ge 2$, by
\[\Delta_m(z) \= \Delta_{m-1}(\widehat{f}(z)) - \Delta_{m-1}(z).\]%\eqref{eq:Delta}.
Let~$\ell$ be the smallest nonnegative integer satisfying~$\ell = \ell_j$ (and thus~$\ell\equiv q  \pmod{p}$). 
%As in~\cite[\emph{Exemple}~3.19]{RiveraLetelier2003} and~\cite{LindahlRiveraLetelier2013}, define recursively for every integer~$m\geq 1$ the power series~$\Delta_m$, by~$\Delta_1(z)\=\widehat{f}(z)-z$ and for~$m \ge 2$, by
%\[\Delta_m(z) \= \Delta_{m-1}(\widehat{f}(z)) - \Delta_{m-1}(z).\]
%To prove the proposition we use
%\begin{equation}
%  \label{eq:7}
%  \Delta_p(z) \equiv \widehat{f}^p(z)-z \mod \langle p \rangle.
%\end{equation}

For each integer~$m \ge 1$ define~$\alpha_m$, and~$\beta_m$ in the ring~$F_1 \= \Z_{(p)}[x_0, x_1]$ by the recursive relations 
\begin{align}
  \alpha_{m+1} & \= x_0 (qm+1)\alpha_m, \label{receq1short}
  \\ 
  \beta_{m+1} & \= x_1(qm+1)\alpha_m + x_0 (qm+\ell_j+1)\beta_m,
                \label{receq2short}
\end{align}
with initial conditions~$\alpha_1 \= x_0$ and~$\beta_1 \= x_1$.
We prove by induction that for every integer~$m \ge 1$ we have
\begin{equation}
  \label{deltaclaimshort}
  \Delta_m(z)
  \equiv
  \alpha_mz^{qm+1} + \beta_mz^{qm+\ell_j+1} \mod \langle z^{qm + \ell_j + 2} \rangle.
\end{equation}
For~$m=1$ this holds by definition. Assume further that it is valid for some~$m\geq1$.
Then
\begin{displaymath}
  \begin{split}
  \Delta_{m+1}(z)
  & = \Delta_m(\widehat{f}(z)) - \Delta_m(z)
  \\ &\equiv \alpha_mz^{qm+1}\left[\left(1 + x_0 z^q + x_1z^{q+\ell_j}  + \cdots\right)^{qm+1} - 1\right]\\
&\qquad + \beta_mz^{qm+\ell_j+1}\left[\left(1 + x_0 z^q + x_1z^{q+\ell_j}  + \cdots\right)^{qm+\ell_j+1}-1\right]\\
&\qquad \mod \langle z^{qm+\ell_j+2} \rangle\\
&\equiv \alpha_m\left[ z^{q(m+1)+1}x_0 (qm+1) + z^{q(m+1)+\ell_j+1}x_1 (qm+1)\right] \\
&\qquad+ \beta_mz^{q(m+1)+\ell_j+1}x_0 (qm+\ell_j+1) \mod \langle z^{q(m+1)+\ell_j+2} \rangle.    
  \end{split}
\end{displaymath}
In view of~\eqref{receq1short} and~\eqref{receq2short}, this proves the induction step and~\eqref{deltaclaimshort}.

By~\eqref{eq:7} and~\eqref{deltaclaimshort}, to prove the  proposition it is sufficient to prove
\begin{equation}
  \label{eq:5short}
  \alpha_p \equiv 0 \mod p F_1
  \text{ and }
  \beta_p \equiv -x_0^{ p - 1 }x_1  \mod p F_1.
\end{equation}
We do this by solving explicitly the linear recurrences described in~\eqref{receq1short}, and \eqref{receq2short}.
By telescoping \eqref{receq1short}, we obtain for every~$m \ge 1$ the solution
\begin{equation}
  \label{alpham}
  \alpha_m = x_0^m \prod_{j=1}^{m-1}(qj+1).
\end{equation}
Taking~$m = p$ we obtain the first congruence in~\eqref{eq:5short}.

On the other hand, inserting \eqref{alpham} in~\eqref{receq2short} yields
\[\beta_{m+1} = x_0^mx_1 \prod_{j=1}^m (qj+1) + x_0 (qm + \ell_j+1)\beta_m.\]
Put~$\widehat{\beta}_1 \= x_1$ and define recursively \[\widehat{\beta}_{m+1} \= x_0^mx_1 \prod_{j=1}^m (\ell j+1) + x_0 (\ell(m + 1)+1)\widehat{\beta}_m.\] We note that for every integer~$m\geq 1$ we have~$\widehat{\beta}_m \equiv \beta_m \mod \langle p \rangle.$
We utilize the substitution
\[\widehat{\beta}^*_m \= \widehat{\beta}_m \bigg/ \left( x_0^{m - 1} \prod_{j=1}^m (\ell j+1) \right),\]
which yields
\[\widehat{\beta}^*_{m+1}
  =
  \widehat{\beta}^*_{m} + \frac{x_1}{\ell(m+1)+1}.\] 
Using~$\widehat{\beta}^*_1 = \frac{x_1}{\ell + 1}$, we obtain inductively for every~$m \ge 1$
\begin{displaymath}
  \widehat{\beta}^*_m
  =
  \sum_{r=1}^{m} \frac{x_1}{\ell r + 1}.  
\end{displaymath}
Equivalently,
\begin{equation}\label{betam}
  \widehat{\beta}_m
  =
  x_0^{m - 1}x_1 \sum_{r=1}^{m}\left[  \prod_{j \in \{1, \ldots, m\} \setminus \{ r \}} (\ell j+1) \right].
\end{equation}
When~$m = p$ every term in the sum above contains a factor~$p$, except for the unique~$r$ in~$\{1, \ldots, p \}$ such that~$\ell r \equiv -1 \pmod{p}$.
Denote by~$r_0$ this value of~$r$, and note that \[\prod_{j \in \{ 1, \ldots, p \} \setminus \{ r_0 \}} (\ell j+1) \equiv -1 \mod p F_1.\]
Hence, we obtain
\begin{displaymath}
  \begin{split}
  \widehat{\beta}_p \equiv \beta_p
  &\equiv
    x_0^{p - 1}  x_1 \prod_{j \in \{ 1, \ldots, p \} \setminus \{ r_0 \}} (\ell j+1) \mod p F_1
  \\ &\equiv
      -x_0^{p - 1} x_1 \mod p F_1.    
  \end{split}
\end{displaymath}
This proves the second congruence in~\eqref{eq:5short} and thus the proposition.
\end{proof}

\begin{proof}[Proof of Theorem \ref{thm:criterion2}]
Suppose that~$j$ is the smallest integer such that~$\pind_j(f) \neq 0$, and put~$\omega(n)\= \ell_j(1 +p+\cdots+p^{n-1}) + qp^n$.  Then by Lemma \ref{prop:ind2conj}~$f$ is conjugated to~$g$ of the form \[g(z) \equiv z(1 + \alpha z^q + \beta z^{q+\ell_j} ) \mod \langle z^{q+\ell_{j+1}+1} \rangle,\] such that~$\beta/\alpha^2 = \pind_j(f)$, and thus we have $\alpha^{p-1}\beta = \alpha^{p+1} \pind_j(f).$ Moreover, since~$i_0(f)=i_0(g)=q$ and~$i_1(f) = i_1(g)$ we have by Proposition \ref{laubiecorr} that~$i_n(f) = \omega(n)$ if~$i_1(f) = i_1(g) = qp +\ell_j$.

Let~$\Z_{(p)}$ and~$F_\infty$ be as in Lemma~\ref{lemma:main}.
  Moreover, let~$h \colon F_\infty \to \K$ be the unique ring homomorphism extending the reduction map~$\Z_{(p)} \to \F_p$, such that~$h(x_0) = \alpha$,~$h(x_1)=\beta$ and such that for every~$i \ge 2$ the element~$h(x_i)$ of~$\K$ is the coefficient of~$z^{q +\ell_{j+1} + i-1}$ in~$g$.
Then~$h$ extends to a ring homomorphism~$F_{\infty}[[ z ]] \to \K [[ z ]]$ that maps~$\widehat{f}$ to~$g$.
So, Proposition~\ref{prop:deltaiteratesshort} implies
\begin{displaymath}
  g^{p}(z) - z
\equiv
-\alpha^{p-1}\beta z^{qp+\ell_j+1}  \equiv -\alpha^{p+1} \pind_j(f)  z^{qp+\ell_j+1} \mod \langle z^{qp+\ell_j+2} \rangle.
\end{displaymath}
This proves that~$i_1(g) = i_1(f) = qp +\ell_j$ if and only if~$\pind_j(f) \neq 0$.
\end{proof}

\begin{proof}[Proof of Corollary \ref{qramiflarge}]
This follows from Theorem \ref{thm:criterion2} and the latter case of Proposition \ref{prop:deltaiteratesshort}. If there is an integer~$j < s$ such that~$\pind_j(f) \neq 0$ then by Theorem \ref{thm:criterion2}~$f$ is not~$q$-ramified. Further, assuming that there is no~$j < s$ such that~$\pind_j(f)\neq0$ then by Lemma \ref{prop:ind2conj}~$f$ is conjugated to \[z\mapsto z(1 + \alpha z^q + \beta z^{2q}) \mod \langle z^{2q+2} \rangle,\] and by Proposition \ref{prop:deltaiteratesshort} with~$x_0 = \alpha$ and~$x_1 = \beta$,~$f$ is~$q$-ramified if and only if \[\alpha^{p+1}\left(\frac{q+1}{2} - \frac{\beta}{\alpha^2}\right) = \alpha^{p+1}\resit(f)\neq 0.\] This completes the proof of the corollary.
\end{proof}

\section{Closed formula for $\pind_j(f)$}\label{app:closed}
For a wildly ramified power series $f$ we can, from the definitions of the second residue fixed point index and its generalizations $\pind_j(f)$, compute a closed formula for these invariants in terms of the first significant coefficients of $f$. In order to state this result, denote by $\mathbb{N}$ the set of nonnegative integers and for an integer~$q \ge 1$ and~$(\MultiInd_0, \ldots, \MultiInd_q)$ in~$\N^{q + 1}$, define
\begin{displaymath}
  |(\MultiInd_0, \ldots, \MultiInd_q)| \= \sum_{j = 0}^{q} \MultiInd_j
  \text{ and }
  \| (\MultiInd_0, \ldots, \MultiInd_q) \| \= \sum_{j = 1}^q j \MultiInd_j.
\end{displaymath}

\begin{prop}
  \label{thm:closed-formula}
  Let $p$ be a prime and $\K$ be a field of characteristic $p$. Furthermore, let $q\geq p+1$ an integer, and~$f$ a power series with coefficients in~$\K$ of the form
  \begin{equation}
    \label{psform22}f(z)
    =
    z\left(1 + \sum_{j=q}^{+\infty} a_jz^j\right), \text{ with } a_q \neq 0.
  \end{equation} %Further, assume that $j$ is the smallest integer such that $\pind_j(f) \neq 0$.
  Then for any integer $\ell_j$ such that $\ell_j = q$ in $\K$ we have 
  \begin{equation}
    \label{eq:closed-formula}
    \pind_j(f)
    =
    - \frac{1}{a_q^{q+1}}\sum_{\substack{\MultiInd \in \N^{q-\ell_j+1} \\ |\MultiInd| = \ell_j, \| \MultiInd \| = \ell_j }} (-1)^{\ell_j-\MultiInd_0}\binom{\ell_j-\MultiInd_0}{\MultiInd_{1},\ldots,\MultiInd_{\ell_j}}\prod_{i = 0}^{\ell_j} a_{q + i}^{\iota_j}.
  \end{equation}
\end{prop}

\begin{proof}
  From the definition, $\pind_j(f)$ is equal to the coefficient of~$\frac{1}{z}$ in the Laurent series expansion about~$0$ of
\begin{equation}
  \label{eq:laurentindf}
  \begin{split}
    \frac{z^{q-\ell_j}}{z-f(z)}
  & =
  -\frac{z^{q-\ell_j}}{a_qz^{q+1}+a_{q+1}z^{q+2}+\cdots+a_{2q}z^{2q+1}+\cdots}
  \\ & =
  -\frac{z^{q-\ell_j}}{a_qz^{q+1}}\cdot\frac{1}{1 + \frac{a_{q+1}}{a_q}z +\frac{a_{q+2}}{a_q}z^2+\cdots}
  \\ &=   -\frac{1}{a_qz^{\ell_j+1}}\cdot\frac{1}{1 + \frac{a_{q+1}}{a_q}z +\frac{a_{q+2}}{a_q}z^2+\cdots}
  \\ & =
  -\frac{1}{a_q^{\ell_j+1}z^{\ell_j+1}}\sum_{i=0}^{+ \infty}a_q^{\ell_j-i}(-1)^i\left(a_{q+1}z +a_{q+2}z^2+\cdots\right)^i.
  \end{split}
\end{equation}
  Thus, $\pind_j(f)$ is equal to the coefficient of~$z^{\ell_j}$ in the sum in \eqref{eq:laurentindf}. We see that for $k\geq q+ \ell_j+1$ the coefficient $a_k$ makes no contributions to the coefficient of $z^{\ell_j}$ in the sum in \eqref{eq:laurentindf}. Furthermore, if $i > \ell_j$ the corresponding term in the sum \eqref{eq:laurentindf} also makes no contributions. Hence, we may replace the sum by \begin{equation}
    \label{eq:sumlittlea}
    \sum_{r=0}^{\ell_j}a_q^r(-1)^{\ell_j-r}(a_{q+1}z+\cdots+a_{q+\ell_j}z^{\ell_j})^{\ell_j-r},
  \end{equation}
and have that $\pind_j(f)$ is equal to the coefficient of $z^{\ell_j}$ of
\[-\frac{1}{a_q^{\ell_j+1}} \sum_{r=0}^{\ell_j}a_q^r(-1)^{\ell_j-r}(a_{q+1}z+\cdots+a_{q+\ell_j}z^{\ell_j})^{\ell_j-r}.\]
 Using the multinomial theorem and regrouping, \eqref{eq:sumlittlea} is equal to
\begin{multline*}
    \sum_{r=0}^{\ell_j}a_q^{r}(-1)^{\ell_j-r}\sum_{\substack{(\MultiInd_1, \ldots, \MultiInd_{\ell_j}) \in \N^{\ell_j} \\ \MultiInd_{1}+\ldots+\MultiInd_{\ell_j} =\ell_j-r}}\binom{\ell_j-r}{\MultiInd_{1},\ldots, \MultiInd_{\ell_j}}\prod_{i=1}^{\ell_j} (a_{q+i} z^i)^{\MultiInd_{i}}
    \\ =
    \sum_{\substack{\MultiInd \in \N^{\ell_j + 1} \\ |\MultiInd| = \ell_j}}(-1)^{\ell_j-\MultiInd_0}\binom{\ell_j-\MultiInd_0}{\MultiInd_{1},\ldots,\MultiInd_{\ell_j}} \left(\prod_{i = 0}^{\ell_j} a_{q + i}^{\MultiInd_i}\right)z^{\|\MultiInd\|}.
\end{multline*}
In the last expression, the term in~$z^{\ell_j}$ is given by restricting the sum to those multi-indices~$\MultiInd$ satisfying~$\| \MultiInd \| = \ell_j$.
This proves the proposition. 
\end{proof}

\section{Acknowledgments}
I would like to thank my supervisor Karl-Olof Lindahl for fruitful discussions during the writing of this paper, and Juan Rivera-Letelier for helpful comments. I would also like to thank Matteo Ruggiero for helpful conversations.

\appendix

\section{Periodic points of wildly ramified power series with large multiplicity}\label{app:perpoints}
In this section we give a proof of Theorem \ref{thm:lower-bound}. Let $(\K, |\cdot|)$ be an non-archimedean field of positive characteristic $p$. Denote by $\widetilde{\K}$ its residue class field $\mathcal{O}_{\K} / \mathfrak{m}_{\K}$, and for an element~$a$ of~$\mathcal{O}_{\K}$, denote by the~$\widetilde{a}$ its reduction in~$\widetilde{\K}$.
The reduction of a power series~$f$ in~$\mathcal{O}_{\K}[[z]]$, is the power series~$\widetilde{f}$ in~$\widetilde{\K}[[z]]$ whose coefficients are the reductions of the corresponding coefficients of~$f$.
For a power series~$f$ in $\mathcal{O}_{\K}[[z]]$, the \emph{Weierstrass degree} $\wideg(f)$ of~$f$ is the order in $\widetilde{\K}[[z]]$ of the reduction $\widetilde{f}$ of~$f$.  

 The proof of Theorem \ref{thm:lower-bound} follows after the following definition and lemma.%If the characteristic~$\widetilde{\K}$ is positive, and~$f$ is a wildly ramified power series in~$\mathcal{O}_{\K}[[z]]$, it is well-known that the minimal period of every periodic point of~$f$ in~$\mathfrak{m}_{\K}$ is a power of~$p$, see for instance \cite[Lemma 2.3]{LindahlRiveraLetelier2015}.

\begin{mydef}
  Let $p$ be a prime number and $\K$ field of characteristic $p$.
  For a wildly ramified power series~$f$ in $\K[[z]]$, define for each integer $n \geq 0$ the element~$\delta_n(f)$ of $\K$ as follows: Put $\delta_n(f) \= 0$ if $i_n(f) = +\infty$, and otherwise let $\delta_n(f)$ be the coefficient of~$z^{i_n(f)+1}$ in~$f^{p^n}(z)$.
\end{mydef}

\begin{lemma}[Special case of Lemma~2.4 in \cite{LindahlRiveraLetelier2015}]\label{lem24}
  Let~$p$ be a prime number and $(\K,|\cdot|)$ an non-archimedean field of characteristic $p$.
  Then, for every wildly ramified power series~$f$ in $\mathcal{O}_{\K}[[z]]$, the following properties hold.
\begin{enumerate}
\item
  Let~$w_0$ in $\mathfrak{m}_{\K}$ be a fixed point of~$f$ different from~$0$.
  Then we have
  \begin{displaymath}
    |w_0|\geq |\delta_0(f)|
  \end{displaymath}
  with equality if and only if
  \begin{displaymath}
    \wideg(f(z)-z)= i_0(f)+2.
  \end{displaymath}
\item
  Let $n\geq1$ be an integer and~$z_0$ in~$\mathfrak{m}_{\K}$ a periodic point of~$f$ of minimal period~$p^n$.
  If in addition $i_n(f)<+\infty$, then we have
  \begin{displaymath}
    |z_0|
    \geq
    \left|\frac{\delta_n(f)}{\delta_{n-1}(f)}\right|^{\frac{1}{p^n}},
  \end{displaymath}
  with equality if and only if
  \begin{equation}
    \label{widegzeta}
    \wideg\left(\frac{f^{p^n}(z)-z}{f^{p^{n-1}}(z)-z}\right) = i_n(f)-i_{n-1}(f)+p^n.
  \end{equation}
  Moreover, if (\ref{widegzeta}) holds, then the cycle containing $z_0$ is the only cycle of minimal period $p^n$ of~$f$ in $\mathfrak{m}_{\K}$, and for every point $z_0'$ in this cycle $|z_0'|=\left|\frac{\delta_n(f)}{\delta_{n-1}(f)}\right|^{\frac{1}{p^n}}$.
\end{enumerate}
\end{lemma}

It should be noted that the following proof is nearly verbatim that of \cite[Theorem 3]{NordqvistRL2019} with replacing $\resit(f)$ by $\pind_1(f)$.
\begin{proof}[Proof of Theorem~\ref{thm:lower-bound}]
  The assertion about fixed points is a direct consequence of the fact that ${\delta_0(f) = a}$ and Lemma~\ref{lem24}(1).
  
  To prove the statement about periodic points that are not fixed, note first that the statement holds trivially in the case~$\pind_1(f) = 0$.
  Thus, we assume that~$\pind_1(f) \neq 0$, and therefore~$f$ has lower ramification numbers of the form \[i_n(f) = \ell(1+p+\dots+p^{n-1})+qp^n\] by  Theorem~\ref{thm:criterion}.
  In particular, for every integer~$n \ge 1$ we have~$i_n(f) < + \infty$.
  Further, by Proposition~\ref{prop:pn} we have for every integer~$n \ge 1$ %%% XX i_n(g) --> i_n(f)
  \begin{displaymath}
    \delta_n(f) = (-1)^n\alpha^{\frac{p^{n+1}-1}{p-1}}\pind_1(f)^{\frac{p^n-1}{p-1}}.
  \end{displaymath}
  Hence, by Lemma~\ref{lem24}(2) we have for every periodic point~$z_0$ in~$\mathfrak{m}_k$ of minimal period~$p^n$,
\begin{equation}\label{normbound2}
  |z_0|
  \geq
  \left|\frac{\delta_n(f)}{\delta_{n-1}(f)}\right|^{\frac{1}{p^n}}
  =
  \left| a^{p^n} \pind_1(f)^{p^{n-1}} \right|^{\frac{1}{p^n}}
  =
  |a| \cdot | \pind_1(f)|^{\frac{1}{p}}.
\end{equation}
This completes the proof of Theorem~\ref{thm:lower-bound}.
\end{proof}


\begin{thebibliography}{\'{E}81b}

\bibitem[AK09]{AnashinKhrennikov2009}
Vladimir Anashin and Andrei Khrennikov.
\newblock {\em Applied algebraic dynamics}.
\newblock Walter De Gruyter, Berlin, 2009.

\bibitem[Ben19]{Benedetto2019}
Robert~L. Benedetto.
\newblock {\em Dynamics in one non-archimedean variable}, volume 198 of {\em
  Graduate Studies in Mathematics}.
\newblock American Mathematical Society, Providence, RI, 2019.

\bibitem[Cam00]{Camina2000}
Rachel Camina.
\newblock The {N}ottingham group.
\newblock In {\em New horizons in pro-{$p$} groups}, volume 184 of {\em Progr.
  Math.}, pages 205--221. Birkh\"{a}user Boston, Boston, MA, 2000.

\bibitem[\'{E}81a]{Ecalle1981b}
Jean \'{E}calle.
\newblock {\em Les fonctions r\'{e}surgentes. {T}ome {I}}, volume~5 of {\em
  Publications Math\'{e}matiques d'Orsay 81 [Mathematical Publications of Orsay
  81]}.
\newblock Universit\'{e} de Paris-Sud, D\'{e}partement de Math\'{e}matique,
  Orsay, 1981.
\newblock Les alg\`ebres de fonctions r\'{e}surgentes. [The algebras of
  resurgent functions], With an English foreword.

\bibitem[\'{E}81b]{Ecalle1981}
Jean \'{E}calle.
\newblock {\em Les fonctions r\'{e}surgentes. {T}ome {II}}, volume~6 of {\em
  Publications Math\'{e}matiques d'Orsay 81 [Mathematical Publications of Orsay
  81]}.
\newblock Universit\'{e} de Paris-Sud, D\'{e}partement de Math\'{e}matique,
  Orsay, 1981.
\newblock Les fonctions r\'{e}surgentes appliqu\'{e}es \`a l'it\'{e}ration.
  [Resurgent functions applied to iteration].

\bibitem[Joh88]{johnson1988}
D.~L. Johnson.
\newblock The group of formal power series under substitution.
\newblock {\em Journal of the Australian Mathematical Society. Series A. Pure
  Mathematics and Statistics}, 45(3):296--302, 1988.

\bibitem[Kea92]{Keating1992}
Kevin Keating.
\newblock {Automorphisms and extensions of k((t)).}
\newblock {\em Journal of Number Theory}, 41(3):314--321, 1992.

\bibitem[Kea05]{Keating2005}
Kevin Keating.
\newblock How close are {$p$}th powers in the {N}ottingham group?
\newblock {\em J. Algebra}, 287(2):294--309, 2005.

\bibitem[KK16]{KallalKirkpatrick2019}
Kenz {Kallal} and Hudson {Kirkpatrick}.
\newblock {Ramification of Wild Automorphisms of {L}aurent Series Fields}.
\newblock {\em arXiv e-prints}, page arXiv:1611.01077, Nov 2016.

\bibitem[LMS02]{LaubieMovahhediSalinier2002}
Fran{\c{c}}ois Laubie, Abbas Movahhedi, and Alain Salinier.
\newblock {Syst{\`{e}}mes dynamiques non archim{\'{e}}diens et corps des
  normes.}
\newblock {\em Compositio Mathematica}, 132(1):57--98, 2002.

\bibitem[LN17]{NordqvistLindahl2017}
Karl-Olof Lindahl and Jonas Nordqvist.
\newblock Geometric location of periodic points of 2-ramified power series.
\newblock arXiv:1705.08630, 2017.

\bibitem[LN18]{LindahlNordqvist2018}
Karl-Olof Lindahl and Jonas Nordqvist.
\newblock Geometric location of periodic points of 2-ramified power series.
\newblock {\em J. Math. Anal. Appl.}, 465(2):762--794, 2018.

\bibitem[LRL16a]{LindahlRiveraLetelier2015}
Karl-Olof Lindahl and Juan Rivera-Letelier.
\newblock Generic parabolic points are isolated in positive characteristic.
\newblock {\em Nonlinearity}, 29(5):1596--1621, 2016.

\bibitem[LRL16b]{LindahlRiveraLetelier2013}
Karl-Olof Lindahl and Juan Rivera-Letelier.
\newblock Optimal cycles in ultrametric dynamics and minimally ramified power
  series.
\newblock {\em Compos. Math.}, 152(1):187--222, 2016.

\bibitem[LS98]{LaubieSaine1998}
Fran{\c{c}}ois Laubie and M.~Sa\"{i}ne.
\newblock {Ramification of Some Automorphisms of Local Fields}.
\newblock {\em Journal of Number Theory}, 72(2):174--182, 1998.

\bibitem[Mil06]{Milnor2006}
John Milnor.
\newblock {\em Dynamics in One Complex Variable}.
\newblock Princeton University Press, Princeton, N.J., 3. ed. edition, 2006.

\bibitem[Nor17]{Fransson2017}
Jonas Nordqvist.
\newblock Characterization of 2-ramified power series.
\newblock {\em J. Number Theory}, 174:258--273, 2017.

\bibitem[NR19]{NordqvistRL2019}
Jonas {Nordqvist} and Juan {Rivera-Letelier}.
\newblock {Residue fixed point index and wildly ramified power series}.
\newblock {\em arXiv e-prints}, page arXiv:1904.04494, Apr 2019.

\bibitem[RL03]{RiveraLetelier2003}
Juan Rivera-Letelier.
\newblock {Dynamique des fonctions rationnelles sur des corps locaux.}
\newblock {\em Ast{\'{e}}risque}, 287(xv):147--230, 2003.

\bibitem[Rug15]{ruggiero2015}
Matteo Ruggiero.
\newblock Classification of one-dimensional superattracting germs in positive
  characteristic.
\newblock {\em Ergodic Theory Dynam. Systems}, 35(7):2242--2268, 2015.

\bibitem[Sen69]{Sen1969}
Shankar Sen.
\newblock {On automorphisms of local fields}.
\newblock {\em Ann. of Math. (2)}, 90:33--46, 1969.

\bibitem[Sil07]{Silverman2007}
Joseph~H. Silverman.
\newblock {\em The arithmetic of dynamical systems}.
\newblock Springer, New York, N.Y., 2007.

\bibitem[Spe11]{Spencer2011}
Matthew~G. Spencer.
\newblock {\em Moduli Spaces of Power Series in Finite Characteristic}.
\newblock PhD thesis, Brown University, 2011.

\bibitem[Vor81]{Voronin1981}
S.~M. Voronin.
\newblock Analytic classification of germs of conformal mappings {$({\bf
  C},\,0)\rightarrow ({\bf C},\,0)$}.
\newblock {\em Funktsional. Anal. i Prilozhen.}, 15(1):1--17, 96, 1981.

\bibitem[Win04]{Wintenberger2004}
Jean-Pierre Wintenberger.
\newblock {Automorphismes des corps locaux de caract{\'{e}}ristique p}.
\newblock {\em J. Th{\'{e}}or. Nombres Bordeaux}, 16(2):429--456, 2004.

\end{thebibliography}
\end{document}